\documentclass[review]{elsarticle}

\usepackage{amsthm}
\usepackage{amsmath}
\usepackage{amssymb}

\newtheorem{lemma}{Lemma}[section]
\newtheorem{corollary}{Corollary}[section]
\newtheorem{proposition}{Proposition}[section]
\newtheorem{theorem}{Theorem}[section]
\newtheorem{problem}{Problem}[section]
\newtheorem{definition}{Definition}[section]

\begin{document}

\begin{frontmatter}

\title{Further results on equivalence of multivariate polynomial matrices}

\author[mymainaddress]{Jiancheng Guan}

\author[mymainaddress]{Jinwang Liu \corref{mycorrespondingauthor}}
\cortext[mycorrespondingauthor]{Corresponding author}
\ead{jwliu64@aliyun.com}

\author[mymainaddress]{Dongmei Li}
\author[mymainaddress]{Tao Wu}

\address[mymainaddress]{School of Mathematics and Computing Science, Hunan University of Science and Technology, Xiangtan 411201, Hunan, China.}


\begin{abstract}
This paper investigates equivalence of square multivariate polynomial matrices
with the determinant being some power of a univariate irreducible polynomial.
We first generalized a global-local theorem of Vaserstein.
Then we proved these matrices are equivalent to their Smith forms by the generalized global-local theorem.
\end{abstract}

\begin{keyword}
Multivariate polynomial matrices \sep equivalence of polynomial matrices \sep Smith forms
\end{keyword}

\end{frontmatter}

\section{Introduction}

During recent decades the importance of multidimensional systems has been recognized,
due to the applications in areas such as iterative learning control systems
and image and video processing etc. (see \cite{Bose2003}).
Equivalence of systems is an important research problem in the multidimensional system theory.

A multidimensional system can be represented by a multivariate polynomial matrix.
Thus, equivalence of systems is closely related to equivalence of multivariate polynomial matrices.

It is well known that for univariate polynomial matrices, the equivalence problem has been solved.
However, for multivariate polynomial matrices, the problem is still open.
For results on equivalence of bivariate polynomial matrices, see \cite{Li2019, Zheng2023, Lu2023}.
For results on equivalence of multivariate polynomial matrices, see \cite{Lin2006, Liu2017, Li-Liang2020, Li-Liu-Chu2022, Li-Liu-Zheng2020, Lu2020}.

Liu et al.\textsuperscript{\cite{Liu2022}} investigated
equivalence of a class of triangular polynomial matrices and their Smith forms.
Let $K$ be a field.
Let $F \in K[x_1, \dots, x_n]^{l \times l}$ with ${\rm det}(F) = x_1 - f(x_2, \dots, x_n)$.
Lin et al.\textsuperscript{\cite{Lin2006}} proved that $F$ is equivalent to ${\rm diag}(1, \dots, 1, {\rm det}(F))$.
Let $F \in K[x_1, \dots, x_n]^{l \times l}$ with ${\rm det}(F) = (x_1 - f(x_2, \dots, x_n))^t$.
Li et al.\textsuperscript{\cite{Liu2017}} proved that $F$ is equivalent to ${\rm diag}(1, \dots, 1, {\rm det}(F))$
if and only if all $(l - 1) \times (l - 1)$ minors of $F$
and ${\rm det}(F)$ generate the unit ideal.
Now, let $F \in K[x_1, \dots, x_n]^{l \times l}$ with
${\rm det}(F) = (x_1 - f_1(x_2, \dots, x_n))^{t_1}(x_2 - f_2(x_3, \dots, x_n))^{t_2}$.
Li et al.\textsuperscript{\cite{Li-Liang2020}} proved that $F$ is equivalent to ${\rm diag}(1, \dots, 1, {\rm det}(F))$
if and only if all $(l - 1) \times (l - 1)$ minors of $F$ generate the unit ideal.

Let $F \in K[x, y]^{l \times l}$ with ${\rm det}(F) = p(x)^t$, where $p(x) \in K[x]$ is an irreducible polynomial.
Li et al.\textsuperscript{\cite{Li2019}} proved that $F$ is equivalent to ${\rm diag}(1, \dots, 1, p(x)^t)$ if and only if
all the $(l - 1) \times (l - 1)$ minors generate the unit ideal.
Zheng et al.\textsuperscript{\cite{Zheng2023}}
generalized this result.
They gave a necessary and sufficient condition under which $F$ is equivalent to its Smith form.
Recently, Guan et al. \textsuperscript{\cite{Guan2024}} partly generalized this result to multivariate polynomial matrices.
They gave the following problem in the paper:
\begin{problem} \label{Pro1}
Let $F \in K[x_1, \dots, x_n]^{2 \times 2}$ with ${\rm det}(F) = p$, where $p \in K[x_1]$ is an irreducible polynomial.
If $I_1(F) = K[x_1, \dots, x_n]$,
then is $F$ equivalent to ${\rm diag}(1, p)$?
\end{problem}
In this paper, we will give an affirmative answer to this problem.
Then we can completely generalize the above result by Zheng et al. to multivariate polynomial matrices.

\section{Preliminaries}

Let $K$ be a field.
$A = K[x_1, \dots, x_n]$ denotes the polynomial ring in variables $x_1, \dots, x_n$ over $K$.
$A^{l \times m}$ denotes the set of all $l \times m$ matrices with entries in $A$.
${\rm diag}(f_1, \dots, f_l)$ denotes the diagonal matrix in $A^{l \times l}$
whose diagonal elements are $f_1, \dots, f_l$, where $f_i \in A$ for $i = 1, \dots, l$.
Let $F \in A^{l \times m}$.
$d_i(F)$ denotes the greatest common divisor of all the $i \times i$ minors of $F$.
We set $d_0(F) = 1$.
${\rm GL}_l(A)$ denotes the set of all $l \times l$ unimodular matrices over $A$.

\begin{definition} [\cite{Lin2001}] \label{def:1.1}
Let $F \in A^{l \times m}$.
Then $F$ is said to be a ZLP matrix if $l \leq m$ and all the $l \times l$ minors generate unit ideal $A$.
$F$ is said to be a ZRP matrix if $l \geq m$ and all the $m \times m$ minors generate unit ideal $A$.
\end{definition}

\begin{definition} [\cite{Lin2001}] \label{def:1.2}
Let $F \in A^{l \times m}$ with rank $r > 0$.
Suppose $F$ has a factorization
\begin{equation} \label{Def1.2(1)}
F = F_1F_2
\end{equation}
where $F_1 \in A^{l \times r}, F_2 \in A^{r \times m}$.
Then (\ref{Def1.2(1)}) is said to be a ZLP factorization if $F_2$ is a ZLP matrix.
\end{definition}

\begin{definition} [\cite{Lin2001}] \label{def:1.3}
Let $F \in A^{l \times m}$ with rank $r > 0$.
Let $i$ be a fixed positive integer less than or equal to $r$.
Let $a_1, \dots, a_{\beta}$ be all the $i \times i$ minors of $F$.
Let $a_j = d_i(F) \cdot b_j$ for $j = 1, \dots ,\beta$.
Then $b_1, \dots, b_{\beta}$ are called the $i \times i$ reduced minors of $F$.
$J_i(F)$ denotes the ideal generated by $b_1, \dots, b_{\beta}$.
$I_i(F)$ denotes the ideal generated by $a_1, \dots, a_{\beta}$.
\end{definition}

\begin{definition} \label{def:1.4}
Let $F, G \in A^{l \times m}$.
Then $F$ is said to be equivalent to $G$
if there exist $U \in {\rm GL}_l(A), V \in {\rm GL}_m(A)$ such that $F = UGV$.
\end{definition}

If $F$ and $G$ is equivalent, we will write $F \sim G$.

\begin{definition} \label{def:1.5}
Let $F \in A^{l \times m}$ with rank $r > 0$.
Let
$$
\Phi_i =
\left\{
\begin{array}{ll}
d_i(F) / d_{i - 1}(F), & 1 \leq i \leq r;\\
0, & r < i < {\rm min}\{l, m\}.\\
\end{array}
\right.
$$
Then the Smith form of $F$ is
$$
S =
\left(
\begin{array}{cc}
{\rm diag}(\Phi_1, \dots, \Phi_r) & 0_{r \times (m - r)}\\
0_{(l - r) \times r} & 0_{(l - r) \times (m - r)}
\end{array}
\right)
$$
\end{definition}

\section{Main results}

Let $B$ be a commutative ring, and $S$ be a multiplicative set in $B$.
Let $f = \Sigma a_ix^i \in B[x]$.
Let $g = \Sigma (a_i / 1)x^i \in B_S[x]$.
Then we say that $f$ localizes to $g$.
Let $F = (f_{ij})_{l \times m} \in B[x]^{l \times m}$.
Let $f_{ij}$ localize to $g_{ij} \in B_S[x]$.
Let $G = (g_{ij})_{l \times m} \in B_S[x]^{l \times m}$.
Then we say that $F$ localizes to $G$.
Let $a / 1 \in B_S$.
If $a / 1 = 0$, then there exists $s \in S$ such that $sa = 0$.
Thus, if $f$ localizes to 0, then there exists $s \in S$ such that $sf = 0$.
Similarly, if $F$ localizes to 0, then there exists $s \in S$ such that $sF = 0$.
The following lemma is Lemma 2.2 on page 104 of \cite{Lam2006}.

\begin{lemma} \label{Lem 1}
Let $B$ be a commutative ring, and $S$ be a multiplicative set in $B$.
Let $\tau (x) \in {\rm GL}_n(B_S[x])$ be such that $\tau (0) = I_n$.
Then there exists a matrix $\hat{\tau} (x) \in {\rm GL}_n(B[x])$
such that $\hat{\tau} (x)$ localizes to $\tau ((s / 1)x)$ (for some $s \in S$), and $\hat{\tau} (0) = I_n$.
\end{lemma}

\begin{proposition} \label{Pro1}
Let $R$ be a commutative ring, and $S$ be a multiplicative set in $R$.
For $F \in R[t]^{l \times m}$, the following statements are equivalent:
\begin{itemize}
\item[(1)] $F(t) \sim F(0)$ over $R_S[t]$;
\item[(2)] there exists $b \in S$ such that $F(t + bx) \sim F(t)$ over $R[t, x]$.
\end{itemize}
\end{proposition}

\begin{proof}
(2) $\Rightarrow$ (1).
Suppose that $U(t, x)F(t + bx)V(t, x) = F(t)$,
where $U(t, x) \in {\rm GL}_l(R[t, x]), V(t, x) \in {\rm GL}_m(R[t, x])$.
Let $t \mapsto 0, x \mapsto b^{-1}t$.
Then $U(0, b^{-1}t)F(0 + bb^{-1}t)V(0, b^{-1}t) = F(0)$.
Since $U(0, b^{-1}t) \in {\rm GL}_l(R_S[t]), V(0, b^{-1}t) \in {\rm GL}_m(R_S[t])$,
we have $F(t) \sim F(0)$ over $R_S[t]$.

(1) $\Rightarrow$ (2).
Take $\sigma_1(t) \in {\rm GL}_m(R_S[t]), \sigma_2(t) \in {\rm GL}_l(R_S[t])$
such that $\sigma_2(t) F(t) \sigma_1(t) = F(0)$.
Let $\tau_1 (t, x) = \sigma_1(t + x)\sigma_1(t)^{-1}, \tau_2 (t, x) = \sigma_2(t)^{-1} \sigma_2(t + x)$.
Then $\tau_1 (t, x) \in {\rm GL}_m(R_S[t, x]), \tau_2 (t, x) \in {\rm GL}_l(R_S[t, x])$.
Thus
\begin{equation} \label{equation1}
\begin{aligned}
\tau_2 (t, x) F(t + x) \tau_1 (t, x) &= \sigma_2(t)^{-1} \sigma_2(t + x) F(t + x) \sigma_1(t + x)\sigma_1(t)^{-1}\\
                                     &= \sigma_2(t)^{-1} F(0) \sigma_1(t)^{-1}\\
                                     &= F(t) \quad ({\rm over} \  R_S[t, x])\\
\end{aligned}
\end{equation}
Since $\tau_1 (t, 0) = \sigma_1(t) \sigma_1(t)^{-1} = I_m, \tau_2 (t, 0) = \sigma_1(t)^{-1} \sigma_1(t) = I_l$,
we can apply Lemma \ref{Lem 1} over $B = R[t]$.
By this lemma, we can find $\hat{\tau_1} (t, x) \in {GL}_m(R[t, x])$
that localizes to $\tau_1(t, (s_1/1)x)$ for some $s_1 \in S$, and $\hat{\tau_1}(t, 0) = I_m$.
Similarly, we can find $\hat{\tau_2} (t, x) \in {GL}_l(R[t, x])$
that localizes to $\tau_2(t, (s_2/1)x)$ for some $s_2 \in S$, and $\hat{\tau_2}(t, 0) = I_l$.
Let $s = s_1s_2$.
Then $\hat{\tau_1} (t, s_2x) \in {GL}_m(R[t, x])$ localizes to $\tau_1(t, (s/1)x)$.
Similarly, $\hat{\tau_2} (t, s_1x) \in {GL}_l(R[t, x])$ localizes to $\tau_2(t, (s/1)x)$.
Let $H(t, x) = \hat{\tau_2} (t, s_1x) F(t + sx) \hat{\tau_1} (t, s_2x) - F(t)$.
In $H(t, x)$, let $x \mapsto 0$.
Then $H(t, x) \mapsto F(t) - F(t) = 0$.
Thus, $H(t, x) = xG(t, x)$ for some $G(t, x) \in R[t, x]^{l \times m}$.
Since $H(t, x)$ localizes to $\tau_2 (t, (s/1)x) F(t + (s/1)x) \tau_1 (t, (s/1)x) - F(t) = 0$,
we have $H(t, x)$ localizes to 0.
Thus, $G(t, x)$ also localizes to 0.
Then there exists $s' \in S$ such that $s'G(t, x) = 0$.
Therefore, $\hat{\tau_2} (t, s_1s'x) F(t + ss'x) \hat{\tau_1} (t, s_2s'x) - F(t) = s'xG(t, s'x) = 0$.
Then $F(t + ss'x) \sim F(t)$.
\end{proof}

Let $I$ be a ideal in $R$.
We denote the radical of $I$ by ${\rm rad}I$.
The proofs of the following two theorems are almost the same as them of theorem 2.4 and theorem 2.5 on page 105 of \cite{Lam2006}, respectively.
Hence, the proofs are omitted here.

\begin{theorem}
Let $R$ be a commutative ring, and $F \in R[t]^{l \times m}$.
Let
\begin{center}
$I_1 = \{ a \in R | F(t) \sim F(0)$ over $R_a[t] \}$,\\
$I_2 = \{ b \in R | F(t + bx) \sim F(t)$ over $R[t, x] \}.$
\end{center}
Then $I_1, I_2$ are ideals in $R$, with $I_1 = {\rm rad}I_2$.
\end{theorem}

\begin{theorem} \label{Thm2}
Let $R$ be a commutative ring, and $F \in R[t]^{l \times m}$.
If $F(t) \sim F(0)$ over $R_{m_1}[t]$ for all maximal ideals $m_1$ in $R$,
then $F(t) \sim F(0)$ over $R[t]$.
\end{theorem}

\begin{lemma} \label{Lem2}
Let $\theta$ be a $K-$algebra automorphism of $A$.
Let $F \in A^{2 \times 2}$.
If $F \sim
\left(
\begin{array}{cc}
1 & \\
  & {\rm det}(F)\\
\end{array}
\right)
$,
then $\theta(F) \sim
\left(
\begin{array}{cc}
1 & \\
  & {\rm det}(\theta(F))\\
\end{array}
\right)
$.
\end{lemma}

\begin{proof}
Let $UFV = {\rm diag}(1, {\rm det}(F))$,
where $U, V \in {\rm GL}_2(A)$.
Then $\theta(U)\theta(F)\theta(V) = {\rm diag}(1, \theta({\rm det}(F)))$.
Since $\theta({\rm det}(F)) = {\rm det}(\theta(F))$,
we are done.
\end{proof}

Let $p \in K[x_1]$ is an irreducible polynomial.
Let $\overline{A} = (K[x_1] / \langle p \rangle)[x_2, \dots, x_n]$.
Let $\pi: A \rightarrow \overline{A}$ be the natural ring homomorphism.
Let $F = (f_{ij})_{l \times m} \in A^{l \times m}$.
We denote $(\pi(f_{ij}))_{l \times m}$ by $\overline{F}$.

\begin{theorem} \label{Thm3}
Let $F \in A^{2 \times 2}$ with ${\rm det}(F) = p$,
where $p \in K[x_1]$ is an irreducible polynomial.
If $I_1(F) = A$,
then $F \sim
\left(
\begin{array}{cc}
1 & \\
  & p\\
\end{array}
\right)$.
\end{theorem}

\begin{proof}
Since $I_1(F) = A$, we have ${\rm rank}(\overline{F}) = 1$ and $I_1(\overline{F}) = \overline{A}$.
By Theorem 3.2 in \cite{Wang2004}, $\overline{F}$ has a ZLP factorization $\overline{F} = G_1G_2$ with $G_2$ being ZLP.
It's clear that $G_1$ is ZRP.
There exist $F_1 \in A^{2 \times 1}, F_2 \in A^{1 \times 2}$ such that $\overline{F_1} = G_1, \overline{F_2} = G_2$ and
$p$ doesn't divide all coefficients of all the elements of $F_1$ and $F_2$ in varibles $x_2, \dots, x_n$.
Let $F_1 = (u_1, u_2)^t, F_2 = (u_3, u_4)$.
Since $\overline{F_1}$ is ZRP, we have $\langle u_1, u_2, p \rangle = A$.
Let $\theta$ be the $K[x_1]-$algebra automorphism induced by
$x_1 \mapsto x_1, x_2 \mapsto x_2 + x_n^{r_2}, \dots x_{n - 1} \mapsto x_{n - 1} + x_n^{r_{n - 1}}, x_n \mapsto x_n$.
Let the leading coefficient of $\theta(u_1)$ in $x_n$ be $f(x_1)$.
Then $p \nmid f(x_1)$.
Thus, $p, f(x_1)$ are coprime.
Let $R = K[x_1, \dots, x_{n - 1}], t = x_n$.
Let $m$ be a maximal ideal of $R$.
If $p \notin m$, then $F$ is invertible over $R_m[t]$.
Thus, $F \sim I_2 \sim F(0)$ over $R_m[t]$.
Now suppose $p \in m$.
Since $p, f(x_1)$ are coprime, we obtain $f(x_1) \notin m$.
Let $v_1 = \theta(u_1), v_2 = \theta(u_2)$.
since $\theta(p) = p$ and $\langle u_1, u_2, p \rangle = A$, we have $\langle v_1, v_2, p \rangle = A$.
Let $I = \langle u_1, u_2 \rangle$.
Since $p \in m$, we have $I + m[t] = R[t]$.
Let $I' = \langle v_1, v_2 \rangle$ over $R_m[t]$.
Then $I' + (mR_m)[t] = R_m[t]$.
The leading coefficient of $v_1$ in $t$ is $f(x_1)$ which is an unit over $R_m[t]$.
By Lemma 1.1 on page 100 of \cite{Lam2006},
we have $I' \cap R_m + mR_m = R_m$.
Then there exists $g \in I' \cap R_m$ such that $g \notin mR_m$.
Let $g = r_1v_1 + r_2v_2$, where $r_1, r_2 \in R_m[t]$.
Let $U =
\left(
\begin{array}{cc}
r_1 & r_2\\
-v_2 & v_1\\
\end{array}
\right)
\in R_m[t]^{2 \times 2}
$.
Then ${\rm det}U = g$.
Thus, $U$ is invertible.
Let $F' = \theta(F)$.
Let $\alpha = (-v_2, v_1)$.
Let $v_3 = \theta(u_3), v_4 = \theta(u_4)$.
Then
$$\overline{\alpha} \overline{F'} = (-\overline{v_2}, \overline{v_1})\overline{F'}
= (-\overline{v_2}, \overline{v_1})(\overline{v_1}, \overline{v_2})^t(\overline{v_3}, \overline{v_4}) = 0
$$
Then $p | \alpha F'$.
Therefore, there exists $G \in R_m[t]^{2 \times 2}$ such that
$UF' = {\rm diag}(1, p)G$.
Since ${\rm det}F' = p$, we have ${\rm det}G = g$.
Then $G$ is invertible.
Thus, $F' \sim {\rm diag}(1, p)$ over $R_m[t]$.
It's clear that $I_1(F') = R[t]$.
Then $I_1(F') = R_m[t]$ over $R_m[t]$.
Since $R_m$ is a local ring,
we have $F'(0) \sim {\rm diag}(1, p)$ over $R_m[t]$.
Then $F' \sim F'(0)$ over $R_m[t]$.
By theorem \ref{Thm2},
$F' \sim F'(0)$ over $R[t]$.
By induction on $n$, we have
$F' \sim {\rm diag}(1, p)$.
By lemma \ref{Lem2}, we have $F \sim {\rm diag}(1, p)$.
\end{proof}

We immediately obtain the following corollary by lemma 3.9 in \cite{Guan2024}

\begin{corollary} \label{Cor1}
Let $B = {\rm diag}(p^{s_1}, \dots, p^{s_k}, p^s, \dots, p^s) \cdot U \cdot
{\rm diag}(\underbrace{1, \dots , 1}_k, p, \dots, p) \in A^{l \times l}$,
where $s_1 \leq \cdots \leq s_k \leq s, p \in K[x_1]$ is an irreducible polynomial
and $U \in {\rm GL}_l(A)$.
If $d_i(B) = p^{s_1 + \cdots + s_i}$ and $J_i(B) = A$ for $i = 1, \dots, k$,
then
$$
B \sim {\rm diag}(p^{s_1}, \dots, p^{s_k}, p^{s + 1}, \dots, p^{s + 1})
$$
\end{corollary}

\begin{lemma} \label{Lem3}
Let $F \in A^{2 \times 2}$.
If ${\rm rank}(\overline{F}) = 1$ and $J_1(\overline{F}) = \overline{A}$,
then there exists $U \in {\rm GL}_2(A)$ such that
$$
UF = {\rm diag}(1, p)G
$$
for some $G \in A^{2 \times 2}$.
\end{lemma}

\begin{proof}
Let $F =
\left(
\begin{array}{cc}
a & b\\
c & d\\
\end{array}
\right)$.

Now, we assume that ${\rm det}(F) \neq 0$.
Let $L = {\rm rowspace}(F)$.
Let $m$ be a maximal ideal in $A$.
If $p \notin m$, then $p$ is invertible in $A_m$.
Thus, $(L : p)_m = L_m : p = L_m$ is free.
Now suppose $p \in m$.
Since ${\rm rank}(\overline{F}) = 1$, we have $p | {\rm det}(F)$.
Let ${\rm det}(F) = ph$.
Since $J_1(\overline{F}) = \overline{A}$, we have $\langle a, b, c, d, p\rangle = A$.
Then at least one of $a, b, c, d \notin m$.
Without loss of generality, assume that $a \notin m$.
Then $a$ is invertible in $A_m$.
We have $F \sim F_1 :=
\left(
\begin{array}{cc}
1 & b_1\\
0 & ph\\
\end{array}
\right)$
over $A_m$ for some $b_1 \in A_m$.
It's clear that $L_m = {\rm rowspace}_{A_m}(F_1)$.
Let $\alpha = (a_1, a_2) \in L_m : p$.
Then $p\alpha \in L_m$.
Let $p\alpha = k_1(1, b_1) + k_2(0, ph)$, where $k_1, k_2 \in A_m$.
Then $p\alpha = (k_1, k_1b_1 + k_2ph) = (pa_1, pa_2)$.
Thus, $pa_1 = k_1$, and $pa_2 = k_1b_1 +k_2ph = pa_1b_1 + k_2ph$.
Then $a_2 = a_1b_1 + k_2h$.
Therefore, $(a_1, a_2) = (a_1, a_1b_1 + k_2h) = a_1(1, b_1) + k_2(0, h)$.
Let $M$ be the submodule of $A_m^2$ generated by $(1, b_1)$ and $(0, h)$.
Then $\alpha \in M$.
Thus, $L_m : p \subseteq M$.
It's clear that $p(1, b_1)$ and $p(0, h) \in L_m$.
Then $M \subseteq L_m : p$.
Therefore, $M = L_m : p$.
It's clear that $(1, b_1), (0, h)$ is a free basis of $M$.
Then $L_m : p$ is free.
Then $L : p$ is projective by Corollary 3.4 on page 19 in \cite{Lam2006}.
By Quillen-Suslin Theorem, $L : p$ is a free module of rank 2.
All the $2 \times 2$ minors of $(pI_2, F)$ are
$p^2, ph, pc, pd, -pa, -pb$.
Since $\langle a, b, c, d, p\rangle = A$, we have $d_2(pI_2, F) = p$.
Then $p$ is regular with respect to $F$ according to definition 2.2 in \cite{Wang2007}.
By theorem 3.1 in \cite{Wang2007}, we have $F = F_1F_2$ for some $F_1$ and $F_2 \in A^{2 \times 2}$ with ${\rm det}(F_1) = p$.
Since ${\rm det}(F_1) = p$, we have $p \in I_1(F_1)$.
Since $I_1(F) \subseteq I_1(F_1)$ and $\langle a, b, c, d, p\rangle = A$, we have $I_1(F_1) = A$.
By theorem \ref{Thm3}, there exist $U, V \in {\rm GL}_2(A)$ such that $UF_1V = {\rm diag}(1, p)$.
Then $UF = UF_1F_2 = {\rm diag}(1, p)V^{-1}F_2$.

Now, suppose that ${\rm det}(F) = 0$.
Note that at least one of $a, b, c, d \neq 0$.
Without loss of generality, assume that $d \neq 0$.
Let $F_1 =
\left(
\begin{array}{cc}
a + p & b\\
c & d\\
\end{array}
\right)$.
Then ${\rm det}(F_1) = pd \neq 0$.
It's clear that $\overline{F_1} = \overline{F}$.
Then ${\rm rank}(\overline{F_1}) = 1$ and $J_1(\overline{F_1}) = \overline{A}$.
By the above proof, there exists $U \in {\rm GL}_2(A), G \in A^{2 \times 2}$ such that
$$
UF_1 = {\rm diag}(1, p)G.
$$
Let $U =
\left(
\begin{array}{cc}
u_{11} & u_{12}\\
u_{21} & u_{22}\\
\end{array}
\right)$.
Then $UF_1 = UF + U
\left(
\begin{array}{cc}
p & 0\\
0 & 0\\
\end{array}
\right)
= UF +
\left(
\begin{array}{cc}
u_{11}p & 0\\
u_{21}p & 0\\
\end{array}
\right)
$.
Then $UF = {\rm diag}(1, p)(G -
\left(
\begin{array}{cc}
u_{11}p & 0\\
u_{21} & 0\\
\end{array}
\right)
)$.
\end{proof}

Combining the above lemma and corollary 3.2 in \cite{Guan2024}, we obtain the following lemma.

\begin{lemma} \label{Lem4}
Let $F \in A^{l \times l}$.
If ${\rm rank}(\overline{F}) = r$ and $J_r(\overline{F}) = \overline{A}$,
then there exists $U \in {\rm GL}_l(A)$ such that
$$
UF = {\rm diag}(\underbrace{1, \dots ,1}_r, p, \dots, p)G
$$
for some $G \in A^{l \times l}$.
\end{lemma}

The proof of the following lemma is almost the same as that of lemma 3.10 in \cite{Guan2024}.
Hence, the proof is omitted here.

\begin{lemma} \label{Lem5}
Let $F \in A^{l \times l}$ with ${\rm det}(F) = p^t$,
where $p \in K[x_1]$ is an irreducible polynomial and $t$ is a positive integer.
Assume that the Smith form of $F$ is ${\rm diag}(p^{s_1}, \dots, p^{s_l})$ with $s_1 \leq \cdots \leq s_l$.
If $F \sim {\rm diag}(p^{s_1}, \dots, p^{s_{k}}, p^s, \dots, p^s) \cdot G$
for some $G \in A^{l \times l}$,
where $s_k \leq s < s_{k + 1}$,
and $J_i(F) = A$ for $i = 1, \dots, k$,
then there exists
$U \in {\rm GL}_l(A)$ such that
$UG = {\rm diag}(\underbrace{1, \dots ,1}_k, p, \dots, p) \cdot G_1$ for some $G_1 \in A^{l \times l}$.
\end{lemma}

\begin{theorem} \label{Thm4}
Let $F \in A^{l \times l}$ with ${\rm det}(F) = p^t$,
where $p \in K[x_1]$ is irreducible and $t$ is a positive integer.
Assume that the Smith form of $F$ is ${\rm diag}(p^{s_1}, \dots, p^{s_l})$
with $s_1 \leq \cdots \leq s_l$.
The following are equivalent:
\begin{itemize}
\item[(1)] $F$ is equivalent to its Smith form;
\item[(2)] $J_i(F) = A$ for $i = 1, \dots, l$.
\end{itemize}
\end{theorem}

\begin{proof}
(2) $\Rightarrow$ (1). It's trivial.

(2) $\Rightarrow$ (1). Since $d_1(F) = p^{s_1}$,
we have $F = {\rm diag}(p^{s_1}, \dots, p^{s_1}) \cdot G_1$ for some $G_1 \in R^{l \times l}$.
If $s_2 = s_1$, then $F = {\rm diag}(p^{s_1}, p^{s_2}, \dots, p^{s_2}) \cdot G_2$,
where $G_2 = G_1$.
If $s_2 > s_1$, since $J_1(F) = A$, by lemma \ref{Lem5},
we have
$G_1 = U_{21} \cdot {\rm diag}(1, p, \dots, p) \cdot G_{21}$,
where $U_{21} \in {\rm GL}_l(A)$ and $G_{21} \in A^{l \times l}$.
Let $B_{11} = {\rm diag}(p^{s_1}, p^{s_1}, \dots, p^{s_1}) \cdot U_{21} \cdot {\rm diag}(1, p, \dots, p)$.
By lemma 3.7 in \cite{Guan2024}, we have $d_1(B_{11}) = p^{s_1}$ and $J_1(B_{11}) = A$.
By corollary 3.1,
there exist $U_{22}, V_{22} \in {\rm GL}_l(A)$ such that
$$
B_{11}
= U_{22} \cdot {\rm diag}(p^{s_1}, p^{s_1 + 1}, \dots, p^{s_1 + 1}) \cdot V_{22}
$$
Let $G_{22} = V_{22}G_{21}$.
Then $F \sim {\rm diag}(p^{s_1}, p^{s_1 + 1}, \dots, p^{s_1 + 1}) \cdot G_{22}$.
Repeat this process for $s_2 - s_1$ times.
We have $F \sim {\rm diag}(p^{s_1}, p^{s_2}, \dots, p^{s_2}) \cdot G_2$.
Similarly, we can obtain $F \sim {\rm diag}(p^{s_1}, \dots, p^{s_l})G_l$.
It's clear that $G_l$ is a unimodular matrix.
Therefore, $F \sim {\rm diag}(p^{s_1}, \dots, p^{s_l})$.
\end{proof}

The proof of the following lemma is similar with that of lemma 3.1 in \cite{Guan2024}.
Hence, the proof is omitted here.

\begin{lemma} \label{Lem6}
Let $F \in A^{l \times m}$, and ${\rm rank}(F) = r > 0$.
If $J_r(F) = A$, then there exist $U \in {\rm GL}_l(A), V \in {\rm GL}_m(A)$ such that $UFV =
\left(
\begin{array}{cc}
G & 0_{r \times (m - r)}\\
0_{(l - r) \times r} & 0_{(l - r) \times (m - r)}\\
\end{array}
\right)
$,
where $G \in A^{ r \times r}$.
\end{lemma}

\begin{corollary} \label{Cor2}
Let $F \in A^{l \times m}$ with ${\rm rank}(F) = r > 0$ and $d_r(F) = p^t$,
where $p \in K[x_1]$ is irreducible and $t$ is a positive integer.
Assume that the Smith form of $F$ is
$
\left(
\begin{array}{cc}
{\rm diag}(p^{s_1}, \dots, p^{s_r}) & 0_{r \times (m - r)}\\
0_{(l - r) \times r} & 0_{(l - r) \times (m - r)}
\end{array}
\right)
$
with $s_1 \leq \cdots \leq s_r$.
The following are equivalent:
\begin{itemize}
\item[(1)] $F$ is equivalent to its Smith form;
\item[(2)] $J_i(F) = A$ for $i = 1, \dots, r$.
\end{itemize}
\end{corollary}

\begin{proof}
By lemma \ref{Lem6}, $F \sim
\left(
\begin{array}{cc}
G & 0_{r \times (m - r)}\\
0_{(l - r) \times r} & 0_{(l - r) \times (m - r)}\\
\end{array}
\right)$
for some $G \in A^{r \times r}$.
It's clear that ${\rm det}(G) = p^t$ and $J_i(G) = A$ for $i = 1, \dots, r$.
By theorem \ref{Thm4}, $G \sim {\rm diag}(p^{s_1}, \dots, p^{s_r})$.
Then $F$ is equivalent to its Smith form.
\end{proof}


\end{document}